\documentclass{article}
\usepackage{amssymb,amsmath,amsthm}
\usepackage{graphicx}
\usepackage{color}

\def\bN{{\mathbb N}}

\def\bR{{\mathbb R}}
\def\bS{{\mathbb S}}
\def\bT{{\mathbb T}}
\def\bZ{{\mathbb Z}}

\def\cD{{\cal D}}
\def\cF{{\cal F}}
\def\cG{{\cal G}}
\def\cH{{\cal H}}
\def\cN{{\cal N}}

\def\Rt{{\bR^2}}

\def\Rm{{\bR^m}}
\def\Rq{{\bR^q}}

\def\vf{\mathfrak X}
\def\T{T}
\def\cs{\hbox{\rm cs\,}}
\def\cod{\hbox{\rm codim\;}}

\def\span{\hbox{\rm span}}

\def\rk{\hbox{\rm rank}}

\def\smpos{{C^\infty_{+}(\bR^2)}}

\newtheorem{theorem}{Theorem}[section]
\newtheorem*{thmN1}{Theorem N1}
\newtheorem*{thmN2}{Theorem N2}
\newtheorem*{thmIFT}{Theorem IFT}

\newtheorem*{lemW}{Lemma W}

\newtheorem{remark}{Remark}
\newtheorem*{definition}{Definition}
\newtheorem{example}{Example}
\addtocounter{example}{-1}
\newtheorem{lemma}[theorem]{Lemma}
\newtheorem{proposition}[theorem]{Proposition}

\def\det{\mathop{\hbox{det}}}

\def\dim{\mathop{\hbox{dim}}}

\def\dim {\mathop{\hbox{dim}}}

\newcommand{\real}{\mathbb{R}}

\def\Eq{g_{can}}

\begin{document}

\title{Partially Isometric Immersions and Free Maps}

\author{
Giuseppina D'Ambra\thanks{email: dambra@unica.it}, Roberto De Leo\thanks{email: deleo@unica.it},
Andrea Loi\thanks{email: loi@unica.it} \\ 
Dipartimento di Matematica, Universit\`a di Cagliari, \\
Via Ospedale 72, 09124 Cagliari -- Italy
}

\maketitle

\begin{abstract}\noindent
In this paper we investigate the existence of ``partially'' isometric immersions. These are  
maps $f:M\rightarrow \Rq$ which, for a given Riemannian manifold $M$, are isometries 
on some sub-bundle $\cH\subset TM$. 
The concept of free maps, which is essential in the Nash--Gromov theory of isometric immersions, 
is replaced here by that of $\cH$--free maps, i.e. maps whose restriction to $\cH$ is free.
We prove, under suitable conditions on the dimension $q$  of the Euclidean space, 
that $\cH$--free maps are generic and we provide, for the smallest possible value of $q$, 
explicit expressions for $\cH$--free maps in the following three settings: 
1--dimensional distributions in $\Rt$, Lagrangian distributions of completely integrable systems, 
Hamiltonian distributions of a particular kind of Poisson Bracket.

\end{abstract}


\noindent
{\it{Subj.Class}}: 58A30, 58J60, 53C12, 53C20, 37J35, 53D17 
\section{Introduction and motivation}
\label{sec:intro}
The present paper deals with the solvability of certain ``large'' systems of partial differential
equations (PDEs) that appear 
naturally when one considers the problem of inducing a quadratic form on a given distribution, where 
the term {\em large} refers to the fact that these systems contain much more unknowns than equations.
Our work is partly inspired by earlier works (\cite{D93,DL03,DL07}), all of them related, 
in different guises, to the general program of inducing geometric structures developed,
among many methods for solving the isometric immersion problem, by 
M. Gromov in his seminal book~\cite{Gro86}. 

To trace the origin of this theory, we must go back to the work of Nash who
showed~\cite{Nas56} that every compact $C^\infty$ Riemannian manifold can be isometrically 
immersed in $\Rq$ with $q\geq \frac{1}{2}n(3n+11)$. This fundamental result has been
substantially improved and refined by several authors \cite{Gr70A,Gro86,Gun90}. In particular,  the hard analytical 
part of Nash's original proof (based on a powerful technique generalizing the classical
implicit function theorem) was taken up by Gromov who stated a {\em Nash-type implicit function} theorem for a certain class
of differential operators called {\em infinitesimally invertible operators}  
(for the exact statement, see~\cite{Gro86},  p. 117). This is a result which, besides its importance in itself
(it serves as a basis to develop the whole of M. Gromov's method based on $h$-prinviple), represents 
the starting point to afford the study of many problems about the possibility of inducing arbitrary geometric structures.

\vskip 0.3cm
In this paper, $M$ will always denote a $C^{\infty}$ smooth manifold of dimension $m$ with metric 
$g=g_{\alpha\beta}dx^{\alpha}dx^{\beta}$, $\alpha,\beta=1,\dots,m$.
Let $f:M\rightarrow \Rq$ 
be a $C^\infty$ immersion $f=(f^1, \cdots , f^q)$
and let $\Eq =\delta_{ij}dx^idx^j$ be 
the canonical Euclidean metric on $\Rq$.
The map $f$ induces on $M$ a metric 
tensor given in local coordinates by:
\begin{equation}
f^*g_{can}=\delta_{ij}\frac{\partial f^{i}}{\partial x^\alpha}\frac{\partial f^{j}}{\partial x^\beta}dx^{\alpha}dx^{\beta}.
\end{equation}
Here and throughout the paper
we write the above pull-back $f^*g_{can}$ as $\cD(f)$, where 
$$\cD:C^\infty(M,\Rq)\to\Gamma$$
denotes the {\sl metric-inducing operator}, viewed as a map between the
space \linebreak $C^\infty(M,\Rq)$ of smooth maps $f:M\to\Rq$ and the space $\Gamma$ of smooth quadratic 
differential forms $g$ on $M$.
These function spaces  will always be considered  equipped 
with the $C^\infty$-fine  (also called Whitney) topology. Recall that, if the manifold $M$ is compact, 
the Whitney topology coincides with the ordinary $C^\infty$ topology.

Placed in this setting, the isometric immersion problem \cite{Nas56} asks for a 
solution $f$ to the (inducing) equation
\begin{equation}
\label{eq:systemisom}
\cD(f)=g 
\end{equation}
which, in local coordinates,  requires to solve,   for a given positive-definite symmetric matrix
of functions $g_{\alpha\beta}=g_{\alpha\beta}(x)$ on $M$, the system

\begin{equation}
\label{eq:isomEq}
\frac{\partial f^{i}}{\partial x^\alpha}\frac{\partial f^{j}}{\partial x^\beta}\delta_{ij}=
g_{\alpha\beta},\ i,j=1,\dots,q.
\end{equation}
which consists of $m(m+1)/2$ equations in the $q$ unknowns $f^i$.
One of the main steps in the original proof of Nash's isometric immersion theorem
amounts to inverting algebraically the linearized equations corresponding to the 
system~(\ref{eq:isomEq}).
It turns out that the same idea applies  
 to many other instances of nonlinear PDEs of geometric nature, which, following 
 the general approach indicated in~\cite{Gro86}
can be locally  solved after an appropriate 
infinite dimensional implicit function theorem is established to pass from (solutions of) the 
linearized system to (solutions of) the non-linear system.
Once this first step  is done,  the efficency of Gromov's  method allows to derive specific (local) statements
as direct corollaries of the generalized theorem just by specializing,  to the pertinent differential operator, 
the basic properties of infinitesimally invertible operators and the consequent analytical results
(compare \cite{Gr72}).

The formulation of Gromov's generalized implicit function theorem we need for our application is  stated below 
(see~\cite{Gr70A}):

\begin{thmIFT}
Let $\cD:C^{\infty}(M,\Rq)\to\Gamma$ be a $C^\infty$ differential operator 
 which is infinitesimally invertible over some open subset $U\subset C^{\infty}(M,\Rq)$.
 Then the restriction of $\cD$ to $U$ is an open map.
\end{thmIFT}

The formal definition of infinitesimally invertible operator 
is omitted here as it requires elaborate preliminaries (see~\cite{Gro86}, p.115-116). For our purposes it is enough to say 
that an operator is infinitesimally invertible if its linearization is invertible.

In the  case of the metric inducing operator, the class of those maps to which 
Theorem IFT is applicable consists of the maps called {\em free} by Nash.
Recall that an immersion $f: M\rightarrow \Rq$ is said to be {\sl free} if for all $x\in M$ 
the $m+s_m$ vectors $\{\partial_{\alpha}f^i(x), \partial_{\alpha\beta}f^i(x)\}$ 
are linearly independent (clearly a necessary condition for the existence of free maps 
$f:M\to\Rq$ is that $q\geq m+s_m$).

\begin{thmN1} \rm{(\cite{Nas56} and \cite{Gro86}, p.116)}
{\em The metric inducing operator $\cD:C^{\infty}(M,\Rq)\to\Gamma$ is infinitesimally 
 invertible over the set of free maps $M\rightarrow \Rq$}. 
\end{thmN1}

 The next result, which follows by a straightforward application of Thom's transversality theorem
 is also due to Nash (\cite{Nas56})
(compare  \cite{Gro86}, p.33).

\begin{thmN2}
A generic map  $M\rightarrow \Rq$ is free for
 $q\geq 2m+s_m$, where $s_m=\frac{m(m+1)}{2}$.
\end{thmN2}
 (The above  means that free maps are open and dense among all $C^2$-maps for $q\geq 2m+s_m$).

Let us turn now to the subject matter of the present paper, where we are given a $k$-dimensional 
distribution $\cH$ on a $m$-dimensional smooth manifold $M$.
Denote by $\Gamma_\cH$  the space of quadratic differential forms on $\cH$ and 
by $G_\cH\subset \Gamma_\cH$ the (open) sub-set of the positive ones, i.e. the set of smooth 
Riemannian metrics on $\cH$. 

We are concerned with the following question.
Assume that there exists a smooth map $f_0:M\rightarrow \Rq$ which induces on $M$ a metric  
$g_0\in G_\cH$ (from the canonical Euclidean metric $g_{can}$ on $\Rq$). This means that: \begin{equation}
\label{eq:systemisomH}
f_0^*\Eq|_\cH=g_0\,,
\end{equation}
which can be written as
$$\cD_\cH(f_0)= g_0 ,$$
where  $\cD_\cH:C^\infty(M,\Rq)\to \Gamma_\cH$ denotes  the partial differential operator defined by:
$$f\mapsto \cD_\cH(f)= f^*\Eq|_\cH .$$

We shall be interested in seeing   whether the same is true for any other $g\in G_\cH$ close enough to $g_0$,
namely if  the inducing equation
\begin{equation}
\label{eq:D_H(f)=g}
\cD_\cH(f)=g\;
\end{equation}
is solvable  in a neighbourhood of $g_0$.
In different words, we want to find  an open neighbourhood 
$U_{g_0}\subset G_\cH$ of $g_0$ which is contained in the image of $\cD_\cH$.

The above equation~(\ref{eq:D_H(f)=g}) is easily seen to be equivalent
to a system of PDEs which is, for any metric $\tilde g$ on $M$, 
a subsystem of~(\ref{eq:isomEq}). This is the reason for calling 
{\sl partial isometry} between $(\cH, g)$ and $(T\Rq, \Eq )$ 
any map $f:M\to \Rq$ which is a solution of~(\ref{eq:D_H(f)=g}).
It can be useful to notice that, since the map $f$ is not necessarily 
an immersion, unlike in the case of isometries the quadratic form $\cD(f)$
is not required to be positive-definite on the whole $TM$.

To answer the above question 
we have to study the linearization of the operator $\cD_\cH$. 
This is done in Section~\ref{sec:free}, where we introduce the 
notion of $\cH$-free map $f:M\rightarrow\Rq$ and prove, 
for the operator $\cD_\cH$, 
 the analogous of 
Theorems~N1 and~N2.
This is Theorem~\ref{thm:genericHfree}.
As a  corollary (see Theorem~\ref{thm:NGHfree}) we obtain  that the restriction 
of $\cD_\cH$ to the set of $\cH$-free maps is an open map.

As we shall see from the very definition of $\cH$-freedom,  the condition 
$q\geq k+s_k$ is necessary (but not sufficient) for the existence of $\cH$-free maps 
$f:M\rightarrow\Rq$.
It is then  natural to look for explicit expressions of $\cH$-free maps 
when  $q=k+s_k$. Notice that very few 
natural examples of free maps are known so far (see~\cite{Gro86}, p.12).
In Section~\ref{sec:hfree} we show explicit examples of $\cH$-free maps 
when the distribution $\cH$ is one of the following: 
a one-dimensional distribution on $\Rt$;  the Lagrangian distribution 
of a complete integrable system; the Hamiltonian distribution of a certain
kind of Poisson bracket.\\

We end this introduction by stating the main results contained in this paper. 
Their proofs are postponed  to  Section~\ref{sec:hfree}.

\begin{theorem}
\label{thm:bobtrick}
Let $\cH\subset \T\Rt$ be a one-dimensional distribution of finite type
on $\Rt$. Then there exists a smooth  function $f:\Rt\rightarrow\bR$ such that the map
$F:\Rt\rightarrow\Rt$ given by $F=\psi\circ f$ is an $\cH$-free map
for any  free map $\psi:\bR\rightarrow\Rt$.
\end{theorem}
\begin{theorem}
\label{thm:CIS}
Let $(M^{2n},\Omega)$ be a symplectic manifold admitting a completely integrable system 
$\{I_1,\cdots,I_n\}$, $\cH\subset TM$ the $n$-dimensional Lagrangian distribution 
$\cH=\cap_{i=1}^n\ker dI_i$ and $\cF$ the corresponding Lagrangian foliation. 
Assume that the Hamiltonian vector fields associated to the $I_i$ are all complete 
and that every leaf of $\cF$ has no compact component.
Then it is possible to find $n$ smooth real valued functions $f^i$, $i=1, \dots , n$, 
on $M$ such that the map $F:M\to\bR^{n+s_n}$, $s_n=\frac{n(n+1)}{2}$ defined by
$$F(x)=(\psi_1(f^1(x)),\cdots,\psi_n(f^n(x)),f^1(x)f^2(x),\cdots,f^{n-1}(x)f^n(x))$$
is $\cH$-free for any choice of $n$ free maps  $\psi_i: \bR\rightarrow\bR^2$.
\end{theorem}
\begin{theorem}
\label{thm:MagnBra}
Let $M$ be an $n$-dimensional oriented Riemannian manifold, 
$H=\{h_1,\cdots,h_{n-2}\}$ 
a set of $n-2$ functions functionally independent at every point 
and $\{\cdot,\cdot\}_H$
the corresponding Riemann-Poisson bracket.
If $h\in C^\infty(M)$ is functionally independent from all the 
$h_i$ and $\cH$ is the corresponding Hamiltonian 1-dimensional distribution, 
then there exists a (possibly multivalued) smooth function $f:M\rightarrow\bR$ 
such the smooth map $F:M\to\Rt$ given by $F(x) = \psi(f(x))$
is $\cH$-free for every free map $\psi:A\to\Rt$ where $A=\bR$ if $f$ is single-valued or 
$A=\bS^1$ if $f$ is multivalued. 
\end{theorem}
\section{$\cH$-free maps and the linearization of the operator ${\cal D}_\cH$}
\label{sec:free}
Let $\cH$ be a $k$-dimensional  distribution on $M$, i.e. a vector subbundle of $TM$. 
Fix local  coordinates $(x^\alpha)$ on some 
chart $U\subset M$, $\alpha=1,\cdots,m$, and let $\{\xi_a\}$, $a=1,\cdots,k$, be 
a local trivialization for $\cH$. Let $\{\theta^a,\omega^A\}$, $A=1,\cdots,m-k$, be a dual base 
for the whole $T^*M$ such that $i_{\xi_b}\theta^a=\delta_b^a$ and $i_{\xi_b}\omega^A=0$. 
Then 
$$\cH|_U=\bigcap_{A=1}^{m-k}\ker\omega^A.$$
and the gradient of the components of a smooth map $f=(f^1, \dots , f^q):M\to \Rq$ writes 
$$
df^i = u_A^i\omega^A\oplus v^i_a\theta^a\;,\;\;i=1,\cdots,q
$$
where
$v^i_a=i_{\xi_a} df^i=L_{\xi_a}f^i$ (the $u_A^i$ play no role in what follows). 
Then, in local terms, the restriction to $\cH\subset TM$ of $f^*\Eq=\delta_{ij}df^i\otimes df^j$ 
is given by
\begin{equation}
f^*\Eq|_\cH = \delta_{ij} L_{\xi_a}f^i L_{\xi_b}f^j\;\theta^a\otimes\theta^b
\end{equation}
and 
the equation $\cD_\cH(f)=g$ writes locally as
\begin{equation}  \label{eq:DH(f)}
\delta_{ij} L_{\xi_a}f^i L_{\xi_b}f^j = g_{ab},
\end{equation}
where $g_{ab}=g(\xi_a, \xi_b)$, $a,b=1,\dots , k$.

Let $\cH$ be a distribution on $M$. We say that $f\in C^\infty(M,\Rq)$ is an {\em $\cH$-immersion}
if the restriction to $\cH$ of the tangent map $Tf:TM\to T\Rq$ is injective.
\begin{example}\rm
Take $M=A\times B$, where $A$ and $B$ are smooth manifolds. Consider the two natural 
projections $\pi_A$ and $\pi_B$ on $A$ and $B$ and the corresponding two canonical distributions 
$\cH_A=\ker T\pi_B=TA\oplus\{0\}$ and $\cH_B=\ker T\pi_A=\{0\}\oplus TB$.
A map $f\in C^{\infty}(M, \Rq)$ is a $\cH_A$-immersion iff $f(\cdot,b):A\to\Rq$ 
is an immersion for every $b\in B$. Similarly for $\cH_B$.
\end{example}
\begin{example}\rm
For any fiber bundle $(M,N,\pi,F)$ it is defined the canonical distribution of
vertical vectors $V=\ker T\pi\subset TM$. A smooth map $f:M\to\Rq$ 
is a $V$-immersion iff on every trivialization $U\times F$ of $M$ the map 
$f(u,\cdot):F\to\Rq$ is an immersion for every $u\in U$. 
Let now $A$ be a linear connection on $M$ and let $H$ be the horizontal distribution 
with respect to $A$; then a map $f:M\to\Rq$ is a $H$-immersion iff the 
covariant derivatives $\{\nabla_\mu f^i\}$ are linearly independent on every point 
of $M$.
\end{example}
\begin{proposition}
Let $f\in C^\infty(M,\Rq)$. The quadratic form $\cD_\cH(f)$ 
is positive-definite iff $f$ is a $\cH$-immersion.
\end{proposition}
\begin{proof}
Let   $\{\xi_a\}$ be  a local trivialization for $\cH$. 
Then 
$$Tf(\xi_a)=\partial_\alpha f^i\partial_i\otimes dx^\alpha(\xi_a^\beta\partial_\beta)
=\xi_a^\alpha\partial_\alpha f^i\partial_i=(L_{\xi_a}f^i)\partial_i$$
and the proposition follows.
\end{proof}
\begin{proposition}
Let $\cH\subset TM$ be a $k$-dimensional distribution.
If $q\geq m+k$ the set of  $\cH$-immersions is open and dense
in  $C^\infty(M,\Rq)$.
\end{proposition}
\begin{proof}
A map $f:M\to\Rq$ is a $\cH$-immersion iff the $k\times q$ matrix $(L_{\xi_a} f^i)$
has rank  $k$ at every point. 
Denote by $D:M\to M_{k,q}(\bR)$ the 
corresponding map with $D=(L_{\xi_a} f^i)$.
The image $D(M)$ and the set of non maximal rank matrices whose codimension in   
$M_{k,q}(\bR)$ is  $q-k+1$ \cite{Arn71}
do not intersect  if $m<q-k+1$. 
\end{proof}
Let us consider now a smooth 1-parameter deformation ${g_{_\epsilon}}$ of the metric $g$ on $M$
such that ${g_{_0}}=g$ and assume that there exists a corresponding smooth 1-parameter 
deformation ${f_{_\epsilon}}$ such that  ${f_{_0}=f}$. It follows by (\ref{eq:DH(f)}) that 
$$
\delta_{ij}L_{\xi_a}{f_{_\epsilon}}^i L_{\xi_b}{f_{_\epsilon}}^j=g_{\epsilon, ab}\;, 
$$
where $g_{\epsilon, ab}=g_{\epsilon}(\xi_a, \xi_b)$.
Differentiate with respect to $\epsilon$ and set
$$
\delta f^i=\frac{d f^i_{_\epsilon}}{d\epsilon}\bigg|_{\epsilon=0}\;,\;\;\delta {g}_{ab}=\frac{d g_{\epsilon, ab}}{d\epsilon}\bigg|_{\epsilon=0}
$$
thus obtainining the system of $k(k+1)/2$ PDEs:
$$
\delta_{ij}\big( L_{\xi_a}f^i(x) \delta [L_{\xi_b} f^j(x)] + 
\delta [L_{\xi_a} f^i(x)] L_{\xi_b}f^j(x) \big)=\delta g_{ab}(x).
$$
Since
$$L_{\xi_a}f^i \delta [L_{\xi_b} f^j] = 
L_{\xi_b}[L_{\xi_a}f^i\delta f^j] - L_{\xi_b}L_{\xi_a} f^i \delta f^j$$
by defining $\psi_a(x)=\delta_{ij} L_{\xi_a}f^i(x) \delta f^j(x)$ 
one gets the following equivalent algebraic system in the $q$ unknown $\delta f^j$:
\begin{equation}
\label{eq:mainsystem}
\begin{cases}
  \delta_{ij}\,\,L_{\xi_a} f^i\delta f^j &= \psi_a\cr
  \delta_{ij}\,\,(L_{\xi_a}L_{\xi_b}f^i+L_{\xi_b}L_{\xi_a}f^i) \delta f^j &= L_{\xi_a}\psi_b + L_{\xi_b}\psi_a - \delta g_{ab}\cr
\end{cases}
\end{equation}
where the $\psi_a$ are arbitrary functions.

A sufficient condition for this system to be solvable is that the matrix
\begin{equation}
\label{eq:Dx}
D_{\xi_1,\cdots,\xi_k,f}=
\begin{pmatrix}
L_{\xi_1} f^1&\cdots&L_{\xi_1} f^q\cr
\vdots&\vdots&\vdots\cr
L_{\xi_k} f^1&\cdots&L_{\xi_k} f^q\cr
\noalign{\medskip}
L^2_{\xi_1} f^1&\cdots&L^2_{\xi_1} f^q\cr
\noalign{\medskip}
L_{\xi_1}L_{\xi_2} f^1 + L_{\xi_2}L_{\xi_1} f^1&\cdots&L_{\xi_1}L_{\xi_2} f^q + L_{\xi_2}L_{\xi_1} f^q\cr
\vdots&\vdots&\vdots\cr
L^2_{\xi_k} f^1&\cdots&L^2_{\xi_k} f^q\cr
\end{pmatrix}
\end{equation}
has maximal rank. Equivalently, 
the vectors 
$$
L_{\xi_a}f^i,\;\{L_{\xi_a},L_{\xi_b}\}f=\;L_{\xi_a}L_{\xi_b}f^i+L_{\xi_b}L_{\xi_a}f^i
$$
must be linearly independent. 
Notice that we need to assume $q\geq2$, since the matrix~(\ref{eq:Dx}) has at least two lines.\\

\begin{definition}
Let $\cH\subset TM$ be  a $k$-dimensional  distribution on a smooth manifold $M$. We say that a 
smooth map  
$f:M\rightarrow \Rq$ is {\em $\cH$-free at $x\in M$} if 
(in a neighbourhood $U$ of $x$) there exists a trivialization $\{\xi_a\}$ 
of $\cH$  such that 
$$\rk D_{\xi_1,\cdots,\xi_k,f}=k+s_k, \ s_k=\frac{k(k+1)}{2}\;.$$
\end{definition}

Notice that $\cH$-free maps only can exist  for $q\geq k+s_k$. Moreover every $\cH$-free map 
is an $\cH$-immersion so that  the
quadratic differential form $\cD_\cH(f)$ induced on $\cH\subset TM$ by a $\cH$-free map 
$f:M\to\Rq$ is always positive-definite (i.e. $\cD_\cH(f)$ is a Riemannian metric on $\cH$).

We do not deal here with the question of inducing a given metric on a given  $\cH$.
We shall return to this problem in another paper where the notion of partial isometry will be considered in a more general
context  (see  \cite{Gr70B}).\\

Next proposition shows that the above definition is well posed.
\begin{proposition}
\label{thm:invrank}
The rank of the matrix $D_{\xi_1,\cdots,\xi_k,f}$ given by (\ref{eq:Dx}) does not depend on 
the particular choice of the trivialization of $\cH$. 
\end{proposition}
\begin{proof}
Take another trivialization $\{\zeta_a\}$ of $\cH$ in the same neighbourhood $U(x)$ 
of $x$. Then 
$\zeta_a(x)=\lambda_a^b(x)\xi_b(x)$ for some local section $\lambda_a^b(x)$ of the
frame bundle over $\cH$ and
$$L_{\zeta_a} f = \lambda_a^b L_{\xi_b} f\,,\;\;L_{\zeta_a}L_{\zeta_b} f = 
\lambda_a^c L_{\xi_c}\lambda_b^d L_{\xi_d} f + \lambda_a^c\lambda_b^d L_{\xi_c} L_{\xi_d} f\;.$$
Clearly $\rk(L_{\zeta_a} f^i)=\rk(L_{\xi_a} f^i)$. Hence
$$
\rk
\begin{pmatrix}
  L_{\zeta_a} f\cr
  \{L_{\zeta_a},L_{\zeta_b}\} f\cr
\end{pmatrix}
=
\rk
\begin{pmatrix}
  L_{\xi_a} f\cr
  (\lambda_a^c L_{\xi_c}\lambda_b^d + \lambda_b^c L_{\xi_c}\lambda_z^d) L_{\xi_d} f 
  + \lambda_a^c \lambda_b^d \{L_{\xi_a},L_{\xi_b}\}f
\end{pmatrix}
$$
$$
=
\rk
\begin{pmatrix}
  L_{\xi_a} f\cr
  \lambda_a^c \lambda_b^d \{L_{\xi_a},L_{\xi_b}\}f
\end{pmatrix}
=
\rk
\begin{pmatrix}
  L_{\xi_a} f\cr
  \{L_{\xi_a},L_{\xi_b}\}f
\end{pmatrix}
$$
\end{proof}

\begin{example} \label{ex:1distr}\rm
It is known (see~\cite{Kap40}) that every one-dimensional distribution $\cH$ on
the plane  $\Rt$ is orientable and then  it is the kernel of a 
regular~\footnote{Throughout this paper we call a vector field or a $k$-form {\em regular} 
  if they are different from zero at every point of $M$.}
1-form $\omega$.
The metric induced on $\cH=\ker\omega$ by a map $f:\Rt\to\Rt$,
$f(x,y)=(\alpha(x,y),\beta(x,y))$, is, by Eq.~(\ref{eq:DH(f)}), 
$$\cD_\cH(f)=[(L_\xi \alpha)^2+(L_\xi \beta)^2](*\omega)^2$$
where $*$ is the Euclidean  Hodge operator.
Then $f$ is a $\cH$-immersion iff $L_\xi \alpha$ and $L_\xi \beta$ 
do not vanish simultaneously at any point.
Here the matrix $D_{\xi,f}$ is given by the $2\times2$ matrix
$$D_{\xi,f}=
\begin{pmatrix}
  L_\xi \alpha&L_\xi \beta\cr
  \noalign{\medskip}
  L^2_\xi \alpha&L^2_\xi \beta\cr
\end{pmatrix}\;.$$
Therefore $f$ is $\cH$-free iff there is a regular section $\xi$ of $\cH$ such that  
$$L_\xi \alpha\; L^2_\xi \beta - L_\xi \beta\; L^2_\xi \alpha>0$$ 
on the whole plane $\Rt$.
\end{example}
We show below a few concrete examples of $\cH$-free maps $f:M\rightarrow\Rq$ for $q=k+s_k$  
and either $\dim\cH=1$ or $\cod\cH=1$.

\begin{example}\rm\label{exfree}
 Take a regular vector field $\xi\in\vf(\Rm)$ and $\cH=\span\{\xi\}$. 
 Assume that $\xi$ has a component always different from zero, e.g. $\xi^1=1$.
 Then $L_\xi x^1=1$ and a direct calculation shows that the function $F(x)=\psi(x^1)$ 
 is $\cH$-free for every free map $\psi:\bR\to\Rt$ (for example, one can take $\psi(t)=(t,e^t)$ 
 or $\psi(t)=(\sin t,\cos t)$).
\end{example}
\begin{example}\rm
 Consider a Riemannian manifold $(M,g)$ and a regular vector field $\xi$ on $M$ 
 which is the gradient of some function $f$, i.e. $\xi^\alpha=g^{\alpha\beta}\partial_\beta f$.
 Then $L_\xi f = \|\xi\|^2>0$ and, as in the previous example, $F(x)=\psi(f(x))$ 
 is $\cH$-free (for $\cH=\span\{\xi\}$) for any free map $\psi:\bR\to\Rt$ .
\end{example}
\begin{example}\rm
 Take the two commuting vector fields $\xi_1=(\cos y,-\sin y,0)$ and $\xi_2=(0,0,1)$ 
 and consider the (integrable) distribution $\cH=\span\{\xi_1,\xi_2\}\subset \T\bR^3$. 
 The leaves of this $\cH$ are the direct product of the level sets  $f(x,y)=e^x\sin(y)$
 with the $z$ axis and, the space of leaves being  not Hausdorff, this foliation is 
 not topologically equivalent to the trivial one of $\bR^3$.
 A direct computation gives that $L_{\xi_1}(e^x\cos y)=e^{2x}>0$ and $L_{\xi_2}z=1>0$
 so that the map $f:\bR^3\to\bR^5$ defined by
 $$
 f(x,y,z) = (\psi(e^x\cos y),\varphi(z),ze^x\cos y))
 $$
 is $\cH$-free for every pair of free maps $\psi,\varphi:\bR\to\Rt$.
\end{example}
\begin{example}\rm
 Let  $\omega=y\,dx-dz$ be the canonical contact structure on $\bR^3$
 and $\cH=\ker\omega\subset\T\bR^3$ the corresponding (non-integrable)
 codimension-1 distribution on $\bR^3$. Equivalently,
 $\cH=\span\{\xi_1,\xi_2\}$ for  $\xi_1=(0,1,0)$   and $\xi_2=(1,0,-y)$. 
 The fields  $\xi_1$ and $\xi_2$ are both transversal
 to the level sets of the function $f:\bR^3\to\bR^5$ defined by
 $$
 f(x,y,z) = (y,x,e^y,e^x,z)
 $$
 Here the matrix~(\ref{eq:Dx}) writes:
 $$
 D_{\xi_1,\xi_2,f}=
 \begin{pmatrix}
   1&0&e^y&0&0\cr
   \noalign{\medskip}
   0&1&0&e^x&-y\cos y\cr
   \noalign{\medskip}
   0&0&e^y&0&0\cr
   \noalign{\medskip}
   0&0&0&0&-1\cr
   \noalign{\medskip}
   0&0&0&e^x&0\cr
 \end{pmatrix}
 $$
 and $f$ is $\cH$-free as $\det D_{\xi_1,\xi_2,f}=e^{x+y}>0$.
\end{example}

In the following proposition we state a characterization of $\cH$-freedom which,  in  all respects
(including the argument for its proof), is similar to that of freedom for maps $f:M\rightarrow \Rq$
(compare \cite{Gr70A}, p.45).
\begin{proposition}
 Let $\cH$ be a distribution on a smooth manifold $M$, let $S=S^2\cH^*$ be the set of its 
 symmetric $(0,2)$ tensors and $N=(Tf(\cH))^\perp$ the normal bundle to $Tf(\cH)$ in $\Rq$
 with respect to $\Eq$.
 Then a $\cH$-immersion $f:M\to\Rq$ is $\cH$-free iff the ``Wintergarten map'' 
 $\nu:N\to S$ defined by 
 $$
 \nu_x(n_x)=\{L_{\xi_a},L_{\xi_b}\}f^i(x)\delta_{ij}n_x^j\,\theta^a\otimes\theta^b
 $$ 
 is surjective.
\end{proposition}

The next proposition is the analogous of Theorems~N1 and~N2   for partial isometries:

\begin{theorem} 
\label{thm:genericHfree}
Let $\cH\subset TM$ be a $k$-distribution on $M$, $\dim M=m$. 
The operator $\cD_\cH$ is infinitesimally invertible on the set 
of $\cH$-free maps $f:M\to\Rq$. Moreover, if $q\geq m + k + s_k$,
a generic  map $M\to\Rq$ is  $\cH$-free.
\end{theorem}
\begin{proof}
 The infinitesimally invertibility of $\cD_\cH$ follows directly from the definition 
 of $\cH$-freedom. Observe that a map $f:M\rightarrow \Rq$ is $\cH$-free
 when the image of the map 
 $$
 D_{\xi_1,\cdots,\xi_k,f}:M\to M_{s_k,q}(\bR)
 $$
 is contained in the set of matrices of maximal rank.
 In particular a map is {\em not} $\cH$-free when the image of $D_{\xi_1,\cdots,\xi_k,f}$ 
 intersects the set $\cN_{s_k,q}$ of matrices of non-maximal rank, whose codimension is 
 $c=q+1-s_k$~\cite{Arn71}. For a generic $f$ the image $D_{\xi_1,\cdots,\xi_k,f}(M)$
 and $\cN_{s_k,q}$ are transversal and therefore they do not have points in common
 when $\dim D_{\xi_1,\cdots,\xi_k,f}(M)<\cod\cN_{s_k,q}$, 
 namely for $m<c$. Hence a generic map $f$ is $\cH$-free for $q > m - 1 + s_k$.
\end{proof}

By combining  Theorem \ref{thm:genericHfree}  and Theorem IFT we get as a corollary:

\begin{theorem} 
 \label{thm:NGHfree}
 The restriction of $\cD_\cH$ to the subset of $\cH$-free maps of $C^\infty(M,\Rq)$
 is an open map.
\end{theorem}

\section{Construction of $\cH$-free maps for $q=k+s_k$}
\label{sec:hfree}
The purpose of  this section is to   give  proofs of  theorems
\ref{thm:bobtrick},  \ref{thm:CIS} and 
\ref{thm:MagnBra}.
For all three,  the argument of the proof stems from a
result obtained by   J.L. Weiner in an article ~\cite{Wei88}
where he studies the problem of existence of smooth first integrals
for a direction field on the plane.
Here we give  three different variants (quoted below as Lemma 
\ref{thm:W1},  Lemma \ref{thm:W2} and Lemma \ref{thm:W3})
of  Weiner's result which, in the original formulation,  reads as follows:
\begin{lemW}
\label{lem:Weiner}
Let $h$ be a smooth function on $\Rt$ without critical points and let
$\cH=\ker dh\subset T\bR^2$.
Then there exists a smooth  function 
$f:\Rt\rightarrow\bR$ whose level sets are transverse to $\cH$ at every point.
\end{lemW}

\subsection{One-dimensional distributions on ${\Rt}$}
\label{sec:1dim}
Here is our first generalization of the Lemma W above
where the hypothesis $\cH=\ker dh$ is weakened.

Consider the foliation $\cF$ of the integral leaves of $\cH$, that is $\cH=T\cF$,
and let $\pi:\Rt\to\cF$ be the canonical projection which associates to every point 
of $\Rt$ the leaf of $\cF$ passing through it.
A leaf $s$ of $\cF$ is said a {\em separatrix} if there is another
leaf $s'$ such that every saturated (with respect to $\pi$) neighbourhood of $s$
has non-empty intersection with every saturated neighbourhood of $s'$ (e.g. see Fig.~\ref{fig:p1}).
We call leaves $s$ and $s'$ in such a relation  {\em inseparable leaves} of $\cF$ and denote 
by $I_s$ the set of all leaves of $\cF$ which are inseparable from $s$. 
Finally, we say that a foliation $\cF$ is of {\em finite type} if the set $I_s$ is  
finite for every separatrix $s$ of $\cF$.


\begin{lemma}
 \label{thm:W1}
 Let $\cH$ be any 1-dimensional distribution on $\Rt$  of finite type.
 The following three (equivalent) properties hold true:
 \begin{enumerate}
 \item there exists a smooth  function $f:\Rt\rightarrow\bR$ whose level sets are
   transverse to $\cH$ at every point;
 \item for any regular 1-form   $\omega$
   such that $\cH=\ker\omega$ there 
   exists a smooth  function $f:\Rt\rightarrow\bR$ such that $*(\omega\wedge df)>0$;
 \item for any regular section  $\xi$  of $\cH$ there exists a smooth  function $f:\Rt\rightarrow\bR$
   such that  $L_\xi f>0$.
 \end{enumerate}
\end{lemma}
\begin{proof}
 Let $\cH^\perp$ be the distribution orthogonal to $\cH$ with respect to the Euclidean metric 
 on $\Rt$ and let $\cF$ and $\cG$ be  the foliations such that $T\cF=\cH$ and $T\cG=\cH^\perp$.
 By construction,  $\cF$ and $\cG$ are transversal at every point. Moreover,  they can be
 oriented in such a way that   the induced orientation on 
 $\cH_p\oplus\cH^\perp_p$ coincides with the orientation of $T_p\Rt$
 at every point $p\in\Rt$.

 Consider any point $p\in\Rt$ and let $h_p$ and $g_p$ be the leaves of $\cF$ and $\cG$
 passing through $p$. The set $U_p=\pi^{-1}(g_p)$ is a tubular neighbourhood of $h_p$.
 Let $\xi$ be the unitary (with respect to the Euclidean norm) section  
 of $\cH$ pointing in the positive direction, $\Phi^t_\xi$ 
 its flow and $\phi:\bR\to g_p$ a chart on $g_p$ such that $\phi(0)=p$ and compatible 
 with the orientation of $\cG$.
 It is easy to verify that the map $\psi_p:\Rt\to U_p$, defined as
 $\psi_p(x,t)=\Phi^t_\xi(\phi(x))$, is a chart for $U_p$ compatible with the orientation
 of $\Rt$ and that the leaves $h_p$ and $g_p$ correspond respectively, in this coordinate
 system, to the coordinate axes $x=0$ and $t=0$.

 Let now $\omega$ be a 1-form such that $\cH=\ker\omega$ (cfr. Example~\ref{ex:1distr})
 oriented positively with respect to $\cG$, i.e. such that $\omega(\zeta)>0$ if $\zeta$
 is a regular vector field tangent to $\cG$ and oriented positively.
 In the chart $\psi_p$ the vector field $\xi\in\ker\omega$ equals $\partial_t$ and so
 $\psi_p^*\omega = \alpha(x,t) dx$ ( i.e. $\omega$ has no $dt$ component) with $\alpha(x,t)>0$.

 Next, we define on  $U_p$ the smooth function $f_p:U_p\to\bR$ represented in coordinates 
 by $f_{\psi_p}(x,t):=\psi_p^*f_p(x,t)=l(t)$, where $l$ is any real smooth function 
 equal to 0 for $t\leq-1$, to 1 for $t\geq1$ and strictly increasing for $|t|<1$.
 Clearly, the support of $f_p$ is the set $V_p = \psi_p(\bR\times[-1,1])$.
 We have  
 $$
 \psi_p^*(\omega \wedge df_p)(x,t) =\alpha(x,t)\partial_t f_{\psi_p}(x,t)\,dx\wedge dt
 $$
 so that
 $\psi_p^*(*(\omega \wedge df_p))(x,t) = \alpha(x,t)f_{\psi_p}(x,t)\, *(dx\wedge dt)$ 
 is strictly positive inside  $\bR\times(-1,1)$ and constant elsewhere. 
 In particular, $df_p$ has compact support equal to $V_p$.

 Furthermore, every such function $f_p$ can be extended from $U_p$ to the whole $\Rt$. 
 Indeed, the leaf $h_p$ divides $\Rt$ in two disjoint components and the value of $f_p$ 
 on $h_p$ is strictly in  between 0 and 1. 
 By setting $f_p=0$ on all those points outside $V_p$ which belong to the component 
 where $f_p$ takes the value 0 and $f_p=1$ elsewhere, we get the desired extension  
 $f_p:\Rt\to\bR$. 

 Next, let us (arbitrarily) fix, for every separatrix $s_k\in\cF$, a leaf $g_k\in\cG$  
 so that $g_k\cap s_k\neq\emptyset$ (which implies  $g_k\cap s_{k'}=\emptyset$ for $k\neq k'$).
 Consider, inside every $U_k:=U_{p_k}$, the locally finite covering consisting of the sets 
 $C_{k,i}=\psi_k(\bR\times(i-1,i+1))$, $i\in\bZ$, where $\psi_k:=\psi_{p_k}$. 
 The set  $C=\{C_{k,i}, k\in\bN, i\in\bZ\}$ obviously  covers the whole $\Rt$
 and, since the foliation  $\cF$ is of finite type, for every  $k\in\bN$ there is  a finite number 
 of $k'\in\bN$ such that $U_k\cap U_{k'}\neq\emptyset$. Therefore, $C$ is
 locally finite. We denote by  $p_{k,i}$ the points $\psi_k(0,i)$ and set
 $f_{k,i}=f_{p_{k,i}}$. The series
 $$
 \sum_{k=1}^\infty\sum_{i=-\infty}^\infty 2^{-N(k,i)} f_{k,i}
 $$ 
 (here $N:\bN\times\bZ\rightarrow\bN$ indicates any bijection) converges pointwise 
 to a smooth function $f$. Indeed the $f_{k,i}$ are uniformly bounded 
 and the series $\sum_{k=1}^\infty\sum_{i=-\infty}^\infty D_\alpha f_{k,i}$ of the 
 derivatives of all positive orders is a finite sum since all $D_\alpha f_{k,i}$ 
 have support equal to $C_{k,i}$.

 Finally, we have 
 $$
 *(\omega\wedge df) = \sum_{k=0}^\infty\sum_{i=-\infty}^\infty 2^{-N(k,i)} [ *( \omega \wedge df_{k,i}) ]>0
 $$
 since every point of  $\Rt$ belongs to some $C_{k,i}$, where 
 $*(\omega\wedge df_{k,i})>0$. 
\end{proof}
\vskip 0.1cm

\noindent
{\bf Proof of Theorem \ref{thm:bobtrick}.}
Let $\xi$ be a regular section of the distribution $\cH\subset T\Rt$, $\cH=\ker\omega$ 
such that $i_\xi(*\omega)=1$ and let $U=L_\xi^{-1} \left(\smpos \right)$ denote the set 
of all smooth real valued maps $f$ on ${\real}^2$ such that $L_\xi f>0$ 
(this set is non-empty by the previous lemma). We want to show that, for every free map 
$\psi: \bR\rightarrow \Rt$, the map 
$$F(x,y)=\psi(f(x,y))$$ 
is $\cH$-free.
Take $\psi(t)=(a(t),b(t))$. 
We must verify  that the matrix
$$\left(
\begin{array}{cc}
L_\xi[a(f)]&L_\xi[b(f)]\cr
\noalign{\medskip}
L^2_\xi[a(f)]&L^2_\xi[b(f)]\cr
\end{array}
\right)$$
has rank $2$
(cfr. Example \ref{ex:1distr}).

\par\noindent
A direct computation shows that 
$$\det D_{\xi,F} =\left|
\begin{array}{cc}
L_\xi[a(f)]&L_\xi[b(f)]\cr
\noalign{\medskip}
L^2_\xi[a(f)]&L^2_\xi[b(f)]\cr
\end{array}
\right|
=
$$
\smallskip
$$
=
\left|
\begin{array}{cc}
a'(f)L_\xi f&b'(f)L_\xi f\cr
\noalign{\medskip}
a'(f)L^2_\xi f+a''(f)[L_\xi f]^2&b'(f)L^2_\xi f+b''(f)[L_\xi f]^2\cr
\end{array}
\right|
=
$$
\smallskip
$$
=
[a'(f)b''(f)-b'(f)a''(f)][L_\xi f]^3 \neq 0
$$
which, by hypothesis,   is never zero. In fact   we assumed $L_\xi f>0$ and the first factor cannot  
vanish  since $\psi$ is free.
\qed
\begin{figure}
\begin{center}
\includegraphics[width=5cm]{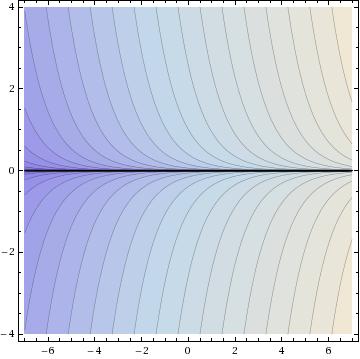}\hskip.75cm\includegraphics[width=5cm]{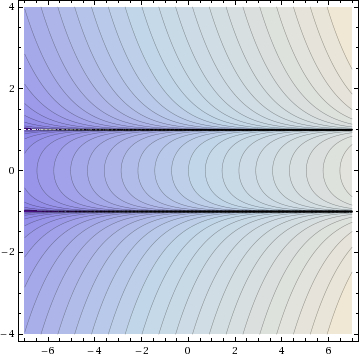}
\end{center}
\caption{%
  \small
  Level sets of the 1-st degree quasi-polynomial regular function $f(x,y)=y e^x$ (left).
  A transversal 2-nd degree quasi-polynomial function $g(x,y)=(y^2-1)e^x$ (right).
  The foliation given by   the level sets of $f$ admits a global transversal (e.g., any 
  vertical straight line).  The foliation given by  the level sets of $g$ does not,  since it admits 
  two separatrices ($y=\pm1$) delimiting a stripe filled by leaves of the kind
  $x=\sec y$.
}
\label{fig:p1}
\end{figure}
\begin{example}\rm
 Consider the distribution $\cH\subset \Rt$ defined by   
 $$
 \cH=\span\{\xi=2y\partial_x+(1-y^2)\partial_y\}\,.
 $$
 This $\cH$ is of finite type since it has only a pair of separatrices, 
 the straight lines $y=\pm1$.
 Indeed $\cH$ is the tangent space to the (Hamiltonian) 
 foliation $\cF$ of the level sets of $f(x,y)=(y^2-1)e^x$
 (a direct computation shows that $L_\xi ((y^2-1)e^x) = 0$).
Moreover, $L_\xi (ye^x) = (1+y^2)e^x>0$,
i.e. the foliation of the level sets of the function $g(x,y)=(1+y^2)e^x$ 
is transverse  to $\cF$ at every point (cfr. Fig.~\ref{fig:p1}).
Then, by Theorem~\ref{thm:bobtrick}, $F(x,y)=\psi(g(x,y))$ is $\cH$-free for 
every free map $\psi:\bR\rightarrow\Rt$.
For example, $F(x,y)=(ye^x,e^{ye^x})$ is $\cH$-free (cfr. Example~\ref{exfree}).
\end{example}
\subsection{The case of completely integrable systems}
\label{sec:CIS}
We can prove another generalization (Lemma~\ref{thm:W2} below)
of Weiner's Lemma in terms of completely integrable systems.


Let $(M^{2n},\Omega)$ be a connected symplectic manifold. 
Since the symplectic 2-form is non-degenerate
it sets up a linear  isomorphism between vector fields $\xi$ and 1-forms 
$\omega$ on $M$ through the relation $i_\xi\Omega=\omega$. 
Moreover, every real valued  function $f:M\rightarrow\bR$ determines 
a unique vector field $\xi_f$ called {\em Hamiltonian vector field}
with the {\em Hamiltonian} $f$ by requiring that for every vector field $\eta$
on $M$ the identity $df(\eta)=\omega (\xi_f, \xi_\eta)$ must hold.
To the given symplectic structure $\Omega$ we can associate,   in a natural way, 
the {\em Poisson bracket} via the formula $\{f, g\}=\Omega (\xi_f, \xi_\eta)$
which turns the algebra $C^{\infty}(M)$ of smooth functions on $M$ into a Poisson algebra.
Assume that $(M^{2n},\Omega)$ admits a regular completely integrable
system. This means that there exists a maximal set of functionally independent Poisson
commuting functions $\{I_i\}$, i.e., such that  $dI_1\wedge\cdots\wedge dI_n\neq0$ at every point of 
$M$ and that the Poisson subalgebra generated by the $I_i$ in $C^\infty(M)$ is abelian.
Consider the distribution $\cH =\ker dI_1\cap\cdots\cap \ker dI_n$ and  
the corresponding Lagrangian foliation  $\cF$ (so that $\cH=T\cF$). 
Then the following theorem, classically known as {\em Arnold--Liouville theorem} holds true
(see~\cite{Arn78} or~\cite{AM78} for details).
%
\begin{theorem}[Arnold--Liouville]
Let $\cF$ the Lagrangian foliation defined above.
If every Hamiltonian vector field $\xi_{I_i}$ is complete then every leaf of $\cF$ 
is diffeomorphic to $\bT^r\times\bR^{n-r}$ and has a saturated neighbourhood $U$ (with
respect to the projection onto the space of leaves $\cF$)
symplectomorphic to the product manifold $D\times(\bT^r\times\bR^{n-r})$, where $D\subset\bR^n$ 
is open, endowed with the coordinates $(I_i,\varphi^j)$ and with the canonical 
symplectic form $\Omega_0=dI_i\wedge d\varphi^i$. 
\end{theorem}
This statement means, in particular, that the commutation relations between the special coordinates
$(I_i,\varphi^j)$ are given by the well-known
$$
\{I_i,I_j\}=0\,,\;\;\{\varphi^i,\varphi^j\}=0\,,\;\;\{I_i,\varphi^j\}=\delta_i^j.
$$
The  $(I_i,\varphi^j)$ are usually called ``action-angle'' coordinates.

When $M=\Rt$  every Hamiltonian system (represented by a single Hamiltonian) is,  
trivially, a completely integrable system. In particular, Lemma~W can be restated as follows:

\vskip 0.3cm

{\em Let  $\{I\}$ be a regular completely integrable system on the symplectic manifold 
$(\Rt,\Omega_0=dx\wedge dy)$. Then there exists a smooth function $f:\Rt\rightarrow\bR$ 
such that  $\Omega_0(\xi_I,\xi_f)>0$ for all points of $\Rt$.}

\vskip 0.3cm

We come now to prove,
based on what we have seen above,
the following Lemma:
\begin{lemma}
\label{thm:W2}
Let $\{I_1,\cdots,I_n\}$ be a regular completely integrable system on $(M^{2n},\Omega)$
and suppose that all the Hamiltonian vector fields $\xi_{I_i}$ are complete. Then there exist $n$ smooth 
functions $\{f^1,\cdots,f^n\}$ (possibly multi-valued) such that 
\begin{equation}
  \label{eq:commRel}
  \{I_i,f^i\}>0\,,\;\;\{I_i,f^j\}=0\,,\;j\neq i
\end{equation}
on the whole manifold $M^{2n}$.
\end{lemma}
\begin{proof}
We follow  closely the original argument in \cite{Wei88}.
By Arnold-Liouville Theorem, every leaf $l\in\cF$ has a saturated neighbourhood 
$U_l\simeq \bR^n\times\bR^k\times\bT^{n-k}$ with coordinates $(I_i,\varphi^j_l)$
such that $U_l$ is defined by the inequalities $\alpha_i^{l}\leq I_i\leq \beta_i^{l}$ 
and
$$
\{I_i,\varphi^j_l\}=\delta^j_i.
$$

We renormalize the action coordinates $I_i$ (which, by hypothesis, are global)
by $J_i^{l}=\mu^{l}_i(I_i-\nu^{l}_i)$ so that $U_l$ is  characterized as
the connected component of $|J_i^{l}|<2$ containing $l$. 
Now, let $b:\bR\to\bR$ be any bump function with support equal to $(-1,1)$ and let  
$l:\bR\to\bR$ be any smooth non-decreasing function which is equal to 0 on 
$(-\infty,-1]$, to 1 on $[1,\infty)$ and strictly increasing between 0 and 1 on $(-1,1)$.

The functions defined on $U_l$ as

$$
f^j_l=b(J^l_1)\cdots b(J^l_n)l(\varphi^j_l)
$$
can be trivially extended to the whole $M$ by setting $f^j_l=0$  outside $U_l$. 
Note that their differentials 
$$
df^i_{l}= \sum_{i=1}^\infty \mu^{l}_i b'(J_i^{l}) dI_i + \sum_{i=1}^\infty b(J_n^{l})l'(\varphi^i_l)d\varphi^i_l
$$
have (modulo the span of the $dI_i$) compact support 
$$V_l=\{p\in U_l|\, |J_i^{l}(p)|<1\,,\;|\varphi^j_l(p)|<1\}\,.$$
Moreover, we have that
$$
\{I_i,f^i_l\}=b(J_1^{l})\cdots b(J_n^{l})l'(\varphi^i_l)>0\,,\;\;\{I_i,f^j_l\}=0\,,\;\;j\neq i
$$
inside $V_l$ while all Poisson brackets are identically zero outside $V_l$.

We extract from the covering $\{U_l\}$ a countable subcovering $\{U_{l_k}\}$.
and show that, by a convenient choice of the coefficients $a_k$, the series
$$f^i=\sum_{k\in\bN}a_k f^i_{l_k}$$
can be made convergent.
In fact the $f^i_{l_k}$ are uniformly bounded so that, by taking $a_k=2^{-k}$,  
the series can be made uniformly convergent. Next,  let us fix any $n$-dimensional 
distribution $\cH'$ transverse to $\cF$ and consider on $M$ the Riemannian metric 
$g=\sum_{i=1}^{n}(dI_i)^2+g'$ (where $g'$ is any metric on $\cH'$). Denote by
$\|D^{(j)}f^i_{l_k}\|$ the norm (associated to the metric $g$) of the derivatives 
of order $j$ of $f^i_{l_k}$. This is seen as a map with domain $M$ and range  the symmetric product 
bundle (of order $j$) $S^jM$ based on $M$. We thus get that, for every value of $k$, 
there is some finite constant $M'_k$ such that, outside $V_{l_k}$, 
$\|D^{(j)}f^i_{l_k}\|\leq M'_{k,i}$ for $1\leq j\leq k$.
Since $V_{l_k}$ has compact closure, there exists another constant $M_{k,i}''$ such 
that $\|D^{(j)}f^i_{l_k}\|\leq M'_{k,i}$ within $V_{l_k}$. This means that 
$\|M^{-1}_{k,i} D^{(j)}f^i_{l_k}\|\leq 1$ for $M_{k,i}=\min\{1,M'_{k,i},M''_{k,i}\}$. 
Therefore, if we take $a_k=2^{-k}M^{-1}_{k,i}$, the series $\sum_{k\in\bN}a_k D^{(j)}f^i_{l_k}$ 
uniformly converges for each $j\in\bN$.

Then  the $f^i$ are smooth and one has
$$
\{I_i,f^i\}>0\,,\;\;\{I_i,f^j\}=0\,,\;\;j\neq i.
$$

Finally, Arnold--Liouville's Theorem tells that the neighbourhoods $U_l$ are all symplectomorphic 
to $\bR^n\times(\bT^r\times\bR^{n-r})$ for some $r$ between $0$ and $n$ and, 
for $r>0$ the leaves have compact components. Observe that, 
on these components,  the $df^j$ are well-defined closed 1-forms. Nevertheless,  these forms
may be non-exact,  due  to the
non-triviality of the first cohomology group of the leaves. Consequently,
in this case the functions $f^j$ may be multivalued, namely   well-defined
only on some covering of $M$.
\end{proof}
\par\noindent
{\bf Proof of Theorem~\ref{thm:CIS}}
We can apply  Lemma \ref{thm:W2} to see that  there exist $n$ smooth functions
$f^i$  satisfying~(\ref{eq:commRel}) which, 
since the leaves of the foliation $\cF$ have no compact component, 
are all single-valued. 
We consider  $n$ free maps  $\{\psi_1,\cdots,\psi_n\}$  from $\bR$ 
to $\Rt$ and prove that the map $F:M\to\bR^{n+s_n}$ defined as
$$
F(x) = (\psi_1(f^1(x)),\cdots,\psi_n(f^n(x)),f^1(x)f^2(x),\cdots,f^{n-1}(x)f^n(x))
$$
is $\cH$-free.

Let $\psi_i(t)=(a_i(t),b_i(t))$ and set $D\psi_i=a_i'b_i''-a_i''b_i'$. The 
square matrix $D_{\xi_1,\cdots,\xi_n,F}$ (see (\ref{eq:Dx}) above) is given, up to 
a permutation of its rows, by
$$
\def\ff#1{g_{#1}}
\left(
\begin{array}{ccccccccccc}
  A_1&\vrule&     *&\vrule&  *&\vrule&        *&\vrule&     *&\vrule&*\cr
  \noalign{\hrule}
    0&\vrule&\ddots&\vrule&  *&\vrule&        *&\vrule&     *&\vrule&*\cr
  \noalign{\hrule}
    0&\vrule&     0&\vrule&A_n&\vrule&        *&\vrule&     *&\vrule&*\cr
  \noalign{\hrule}
    0&\vrule&     0&\vrule&  0&\vrule&2\ff1\ff2&\vrule&     *&\vrule&*\cr
  \noalign{\hrule}
    0&\vrule&     0&\vrule&  0&\vrule&        0&\vrule&\ddots&\vrule&*\cr
\noalign{\hrule}
    0&\vrule&     0&\vrule&  0&\vrule&        0&\vrule&     0&\vrule&2\ff{n-1}\ff{n}\cr
\end{array}
\right)
$$
where $g_i=L_{\xi_{I_i}}f^i$,
$$
A_i=
\begin{pmatrix}
a'_i(f^i)g_i&b'_i(f^i)g_i\cr
\noalign{\medskip}
a'_i(f^i)L^2_{\xi_{I_i}} f^i+a''_i(f^i)g_i^2&b'_i(f^i)L^2_{\xi_{I_i}} f^i+b''_i(f^i)g_i^2\cr
\end{pmatrix}
$$
and the stars represent terms which do not contribute to the determinant.

Since $\det A_i=g_i^3D\psi_i$ and the blocks below the diagonal are identically zero, 
the determinant of $D_{\xi_1,\cdots,\xi_n,F}$   equals  
$$
2^{s_n}\Pi_{k=1}^n(g_i^{n+2}D\psi_k)
$$
which differs from zero at every point because, by construction,   $g_i>0$  and,  
by hypothesis,  $D\psi_i\neq0$. Hence $F$ is a $\cH$-free map.
\qed
\begin{remark}\rm
 Clearly the map $F$ defined in the proof above is modeled after the canonical 
 free map $G:\bR^n\to\bR^{n+s_n}$ given by
 $$
 G(x^1,\dots,x^n)=(x^1,\dots,x^n,(x^1)^2,x^1x^2,\cdots,(x^n)^2)\;.
 $$

 So far it is not known (see~\cite{Gro86}, p.9) whether, for $n\geq2$, there exist 
 free maps from $\bT^n$ to $\bR^{n+s_n}$.
 It is for this reason that in Theorem~\ref{thm:CIS} we require that the leaves 
 of the foliation $\cF$ have no compact component.
 On the other hand, it is an easy matter to check that the map 
 $G:\bT^n\rightarrow \bR^{n+s_n+s_{n-1}}$ defined by
 $$
 G(\theta^1,\dots,\theta^n)
 =(\cs(\theta^1),\dots,\cs(\theta^n),\cs(\theta^1+\theta^2),\dots,\cs(\theta^{n-1}+\theta^n))\;,
 $$
 where $\cs\theta=(\cos\theta,\sin\theta)$, is free.
 \end{remark}
\subsection{The case of Poisson systems}
\label{sec:MagnBra}
Poisson structures are a generalization of symplectic structures having the
nice property of existing even in odd-dimensional manifolds. 
Recall that a Poisson
manifold is a pair $(M,\{,\})$ where $\{,\}:C^\infty(M)\times C^\infty(M)\to C^\infty(M)$ is a 
$\bR$-bilinear skew-symmetric derivation satisfying the Jacobi identity. To every smooth 
function $f\in C^\infty (M)$ it is associated canonically a  {\em Hamiltonian vector field}
$\xi_f$ defined by $\xi_f(g)\stackrel{\rm def}{=}\{f,g\}$.

In particular, when $M=\Rt$, the canonical symplectic form $\Omega_0=dx \wedge dy$ 
induces on $M$ a Poisson Bracket $\{f,g\}=\Omega_0(\xi_f,\xi_g)$ which  can  also be 
obtained via the Euclidean metric as
$$
\{f,g\}=*[df\wedge dg]
$$
where $*$ is the Euclidean Hodge operator. Observe that  this Poisson bracket 
does not need a symplectic structure to be defined but rather an 
orientable Riemannian structure. 
Furthermore,  it can be defined in any dimension
$n$ as follows.
Let $M$ be an oriented Riemannian manifold of dimension $n\geq2$, $*$ its Hodge operator 
and $H=\{h_1,\cdots,h_{n-2}\}$ a set of $n-2$ smooth functions. We set  
$$
\{f,g\}_H\stackrel{\rm def}{=}*[dh_1\wedge\cdots\wedge dh_{n-2}\wedge df\wedge dg]
$$
and call it {\sl Riemann-Poisson} bracket with respect to $H$.
In particular, the foliation corresponding to a Hamiltonian 
vector field $\xi_h$, with $h\in C^\infty(M)$, is given by the intersections of the 
level sets of the $h_i$ with the level sets of $h$.

In the general case each function $h_i$ is what, in the language of Poisson's system,  
is called a Casimir for $\{,\}_H$, meaning that it is an  element of the 
center of $C^\infty(M)$.
\begin{example}\rm
 Let $M=\bR^3$ with the Euclidean metric and coordinates $(x,y,z)$ and let $H=\{x\}$. 
 Then the Riemann-Poisson bracket is given by 
 $\{f,g\}_H = \partial_y f\,\partial_z g - \partial_y g\,\partial_z f$. In particular
 $\xi_y=\partial_z$ and $\xi_z=-\partial_y$ and the coordinate $x$ is a Casimir. 
\end{example}
%
\begin{remark}\rm
 Another example, considered by S.P. Novikov in~\cite{Nov82}, is given when $M$ is the 
 three-torus $\bT^3$ with angular coordinates $(\theta^1,\theta^2,\theta^3)$ and 
 $H=\{h(\theta^i)=B_i\theta^i\}$, $i=1,2,3$, for some constant 1-form $B=B_i d\theta^i$.
 Then the Riemann-Poisson bracket is given by 
 $$
 \{f,g\}_H = \epsilon^{ijk}\partial_i f\,\partial_i g\, B_k\;,
 $$
 where $\epsilon^{ijk}$ is the totally antisymmetric Levi--Civita tensor 
 (for an example of the rich topological structure hidden behind this Riemann-Poisson 
 bracket, the interested reader is referred to~\cite{DD09}).
\end{remark}
Placed in this setting Weiner's Lemma reads as follows:
\vskip 0.3cm
{\em 
Consider the Euclidean plane $\Rt$ endowed with the Riemann-Poisson bracket $\{,\}$ and let  
$h\in C^\infty (\Rt)$ be a regular Hamiltonian. Then there exists $f\in C^\infty (\Rt)$ 
such that $\{h,f\}>0$ on the whole $\Rt$.
}
\vskip 0.3cm
Furthermore,  Weiner's result 
in the latter formulation
can be extended, under a non-degeneracy 
condition,  to  Riemann-Poisson brackets:
\begin{lemma}
\label{thm:W3}  
Let $M$ be an oriented Riemannian manifold of dimension $n\geq 2$ and 
let $H=\{h_1,\cdots,h_{n-2}\}$ be a set of $n-2$ functions functionally 
independent at every point
(i.e. such that $dh_1\wedge\cdots\wedge dh_{n-2}$ never vanishes).
Then, for any $h\in C^\infty(M)$ functionally independent from the $h_i$, 
there exists a smooth function (possibly multivalued) $f:M\to\bR$ such that 
the Riemann--Poisson bracket $\{h,f\}_H$ is strictly positive at every point.
\end{lemma}
\begin{proof}
 Set $h_{n-1}=h$. Let $\cF$ the 1-dimensional Hamiltonian foliation associated 
 to $h$, namely the one defined by $dh_1=\cdots=dh_{n-1}=0$, and let $\pi:M\to\cF$ be
 the canonical associated projection. At every point $p\in M$ there exists a saturated
 (with respect to $\pi$) neighbourhood $U_p\simeq D\times X$, where $D\simeq\bR^{n-1}$ 
 and $X$ is either $\bR$ or $\bS^1$, defined as the connected component of the set
 $W_p=\{a_i< h_i< b_i, i=1,\dots,n-1\}$ which contains $p$. We renormalize these 
 coordinates by using new ones $\hat h_i = \mu_i(h_i-\nu_i)$ so that $W_p$ 
 is defined by $|\hat h_i| < 2$, $i=1,\dots,n-1$.

 Let now $A_p$ be the subset of $U_p$ defined by $|\hat h_i| < 1$, $i=1,\dots,n-1$
 and take two functions $b$ and $l$ like in the proof of Lemma~\ref{thm:W2}.

 The real-valued function 
 $$
 f_p(h_1,\dots,h_{n-1},\varphi) =  l(\varphi )\times\prod_{i=1,\dots,n-1}b(h_i)
 $$
 is well-defined and smooth in $A_p$ and it can be extended to a smooth function on 
 the whole $M$ by setting it equal to zero outside $A_p$. Clearly
 $$
 df_p=\omega_{n-1}\oplus\Pi_{i=1}^{n-1}b(h^i) l'(\varphi)d\varphi
 $$
 where $\omega_{n-1}\in\span\{dh_1,\cdots,dh_{n-1}\}$.
 Then $\{h,f_p\}_H$ everywhere vanishes except within $B_p=\{p'\in A_p|\,|\varphi(p')|<1\}$, 
 where we have 
 $$
 \vbox{
   \halign{$#$\hfill&$#$\hfill\cr
     \{h,f_p\}_H&=*[dh_1\wedge\cdots\wedge dh_{n-1}\wedge df_p]\cr
     &=*[dh_1\wedge\cdots\wedge dh_{n-1}\wedge \Pi_{i=1}^{n-1}b(h_i)l'(\varphi)d\varphi]\cr
     &=*[dh_1\wedge\cdots\wedge dh_{n-1}\wedge d\varphi]\,\prod_{i=1}^{n-1}b(h_i)l'(\varphi)>0\,,\cr
   }
 }
 $$
 the function $*[dh_1\wedge\cdots\wedge dh_{n-1}\wedge d\varphi]$ 
 being positive for all points $q\in U_p$ and every $p\in M$ since $M$ is oriented.

 Now extract a countable subcovering $\{A_{p_k}\}_{k\in\bN}$ from $\{A_p\}$ and
 let $f_k=f_{p_k}$ be the correponding function on every $A_k:=A_{p_k}$.
 As in Lemma~\ref{thm:W2}, the series $\sum_{k} a_k f_k$ can be made convergent to a 
 smooth function $f$ by choosing a convenient sequence $a_k$. Then $\{h,f\}_H>0$ on the 
 whole $M$ since every point $p$ is covered by at least one $A_k$, so that 
 $\{h,f\}_H\geq\{h,f_k\}_H>0$.
\end{proof}
\noindent
{\bf Proof of Theorem \ref{thm:MagnBra}}
Let $\xi_h$ be the Hamiltonian vector field associated to $h$ through $\{,\}_H$.
Then, by Lemma~\ref{thm:W3}, there exists a function $f$ (possibly multi-valued) 
such that $L_{\xi}f>0$. Hence, as seen in Theorem~\ref{thm:bobtrick}, 
the smooth map $F:M\to\Rt$ given by $F(x)=\psi(f(x))$ is $\cH$-free 
($\cH=\span\{\xi_h\}$), where $\psi:\bR\rightarrow\Rt$ (respectively 
$\psi:\bS^1\rightarrow\Rt$) is free if $f$ is single-valued (respectively 
multi-valued). 
\qed

\vskip 0.3cm

\providecommand{\bysame}{\leavevmode\hbox to3em{\hrulefill}\thinspace}
\providecommand{\MR}{\relax\ifhmode\unskip\space\fi MR }
\providecommand{\MRhref}[2]{%
 \href{http://www.ams.org/mathscinet-getitem?mr=#1}{#2}
}
\providecommand{\href}[2]{#2}

\bibliographystyle{amsalpha}

\end{document}